\documentclass[11pt,journal]{IEEEtran}
\usepackage{amsmath,graphicx,bm,amssymb,amsthm,enumerate,cases}
\renewcommand{\(}{\left(}
\renewcommand{\)}{\right)}
\renewcommand{\[}{\left[}
\renewcommand{\]}{\right]}
\renewcommand{\c}{\mathbf{c}}

\renewcommand{\j}{\mathbf{j}}

\newcommand{\w}{\mathbf{w}}

\newcommand{\I}{\mathbf{I}}

\newcommand{\A}{\mathbf{A}}
\newcommand{\T}{\mathbf{T}}

\newcommand{\N}{\mathbb{N}}

\renewcommand{\u}{\mathbf{u}}

\newcommand{\B}{\mathbf{B}}

\renewcommand{\i}{\mathbf{i}}
\newcommand{\Tr}[1]{{\rm{Tr}}\left(#1\right)}
\newcommand{\tr}[1]{{\rm{tr}}\left(#1\right)}

\newcommand{\End}[1]{{\rm{End}}}

\newcommand{\mmod}{\,{\rm{mod}}\;}

\newtheorem{lemma}{Lemma}
\newtheorem{definition}{Definition}

\newtheorem{prop}{Proposition}

\usepackage{fontenc}
\usepackage{inputenc,cuted,bbm}
\usepackage[square,sort,compress,comma,numbers]{natbib}
\usepackage{epstopdf,pst-node}
\usepackage{flushend,stfloats}

\usepackage{enumitem}
\SetLabelAlign{Center}{{\makebox[1em]{\hfil#1\hfil}}}

\newcommand{\mypm}{\mathbin{\smash{%
\raisebox{0.35ex}{%
            $\underset{\raisebox{0.5ex}{$\smash -$}}{\smash+}$%
            }%
        }%
    }%
}

\usepackage{pgfplots,authblk}
\usepackage{smartdiagram}
\usesmartdiagramlibrary{additions}

\usepackage{pgf}

\begin{document}
\title{Asymptotically Pseudo-Independent Matrices}


\author{Ilya Soloveychik and Vahid Tarokh \\ Department of Electrical and Computer Engineering, Duke University
\thanks{This work was supported by the Office of Naval Research grant No. N00014-18-1-2244.}
}
\maketitle

\begin{abstract}
We show that the family of pseudo-random matrices recently discovered by Soloveychik, Xiang, and Tarokh in their work ``Symmetric Pseudo-Random Matrices'' exhibits asymptotic independence. More specifically, any two sequences of matrices of matching sizes from that construction generated using sequences of different non-reciprocal primitive polynomials are asymptotically independent.
\end{abstract}


\section{Introduction}
Random matrices have been a very active area of research for the last few decades and have found enormous applications in various areas of modern mathematics, physics, engineering, biological modeling, and other fields \cite{akemann2011oxford}. In this article, we focus on the classical model of square symmetric matrices with $\mypm 1$ entries, referred to as square symmetric {\it sign} matrices. For this class of matrices, Wigner \cite{wigner1955characteristic, wigner1958distribution} demonstrated that if the elements of the upper triangular part (including the main diagonal) of an $n \times n$ matrix are independent Rademacher ($\mypm 1$ with equal probabilities) random variables, then as $n$ grows a properly scaled empirical spectral measure converges to the semicircular law. 


In many engineering applications, one needs to simulate matrices with \textit{random-looking} properties. The most natural way to generate an instance of a random $n \times n$ sign matrix is to toss a fair coin $\frac{n(n+1)}{2}$ times, fill the upper triangular part of a matrix with the outcomes and reflect the upper triangular part into the lower. Unfortunately, for large $n$ such an approach would require a powerful source of randomness due to the independence condition \cite{gentle2013random}. In addition, when the data is generated by a truly random source, atypical  \textit{non-random looking} outcomes have non-zero probability of showing up. Yet another issue is that any experiment involving tossing a coin would be impossible to reproduce exactly. All these reasons stimulated researchers and engineers from different areas to seek for approaches of generating \textit{random-looking} data usually referred to as \textit{pseudo-random} sources or sequences of binary digits \cite{zepernick2013pseudo, golomb1967shift}. A wide spectrum of pseudo-random number generating algorithms have found applications in a large variety of fields including radar, digital signal processing, CDMA, coding theory, cryptographic systems, Monte Carlo simulations, navigation systems, scrambling, etc. \cite{zepernick2013pseudo}.

The term \textit{pseudo-random} is used to emphasize that the binary data at hand is indeed generated by an entirely deterministic causal process but its statistical properties resemble some of the properties of data generated by tossing a fair coin. Remarkably, most efforts were focused on one dimensional pseudo-random sequences \cite{zepernick2013pseudo, golomb1967shift} due to their natural applications and to the relative simplicity of their analytical treatment. One of the most popular methods of generating pseudo-random sequences is due to Golomb \cite{golomb1967shift} and is based on linear-feedback shift registers capable of generating pseudo-random sequences (also called maximal or $m$-sequences) of very low algorithmic complexity \cite{li2009introduction, downey2010algorithmic}. The study of pseudo-random arrays and matrices was launched around the same time \cite{reed1962note, macwilliams1976pseudo, imai1977theory, sakata1981determining}. Among the known two dimensional pseudo-random constructions the most popular are the so-called perfect maps \cite{reed1962note, paterson1994perfect, etzion1988constructions}, and two dimensional cyclic codes \cite{imai1977theory, sakata1981determining}. However, except for the recent articles \cite{soloveychik2017pseudo, soloveychik2017explicit, soloveychik2017spectral, soloveychik2018symmetric}, to the best of our knowledge none of the previous works considered constructions of symmetric sign matrices using their spectral properties as the defining statistical features. In their work, the authors of \cite{soloveychik2018symmetric} designed a family of symmetric sign $n\times n$ matrices whose spectra almost surely (with respect to a certain ensemble of small size) converge to the semicircular law when their sizes grow. The construction is very simple and is based on binary $m$-sequences of lengths of the form $n=2^m-1$ making the generation of the pseudo-random matrices very efficient and fast. 

The current paper is a natural extension of \cite{soloveychik2018symmetric}. Our goal is to show that the pseudo-random matrices constructed in that article not only yield semicircular spectrum in the limit but also mimic the asymptotic independence properties of the truly random Wigner matrices. We prove that if two sequences of matrices from \cite{soloveychik2018symmetric} are generated using sequences of different non-reciprocal primitive polynomials, then the former are asymptotically independent. Technically, this is achieved by verifying that the mixed centered moments of the matrices at hand vanish asymptotically. This result sheds much more light on the nature of spectral pseudo-randomness in matrices and provides the first example of a family of pseudo-random matrix constructions with the aforementioned design properties and low algorithmic complexity.

The rest of the text is organized as follows. First we set up the notation in Section \ref{sec:not}. Section \ref{sec:gol} introduces the pseudo-random construction defined in \cite{soloveychik2018symmetric} and outlines its properties relevant for the current text. Section \ref{sec:asympt_fr_ps} shows that our pseudo-random matrices are indeed asymptotically independent. Numerical simulations supporting our findings are shown in Section \ref{sec:num}. We make our conclusions in Section \ref{sec:conc}.

\section{Notation}
\label{sec:not}
We denote the ranges of non-negative integers by $[n] = \{0,\dots,n-1\}$. Note also that the labeling of matrix elements starts with $0$. We write $\tr{\A}=\frac{1}{n}\Tr{\A}$, where $\A$ is an $n\times n$ matrix. Introduce a family of functions
\begin{equation}
\label{eq:zeta_def}
\begin{array}{llcl}
\hspace{-0.25 cm} \zeta_n : & GF(2)^{n \times n} & \to & \{-1,1\}^{n \times n}, \\
& \{u_{ij}\}_{i,j=0}^{n-1} & \mapsto & \{(-1)^{u_{ij}}\}_{i,j=0}^{n-1},
\end{array}
\end{equation}
mapping binary $0/1$ matrices into sign matrices of the same sizes. Below we suppress the subscript and write $\zeta$ for simplicity. We use the following standard notation for the limiting relations between functions. We write $f(n)=o(g(n))$ if $\lim_{n\to \infty} \frac{f(n)}{g(n)} = 0$ and $f(n)=O(g(n))$ if $|f(n)| \leqslant C |g(n)|$ for some constant $C$ and $n$ big enough.

\section{The Pseudo-Random Construction}
\label{sec:gol}

In this section, we briefly outline the construction presented in \cite{soloveychik2018symmetric}.

\subsection{Golomb Sequences}
\label{sec:gol_a}
Let $f(x)$ be a binary primitive polynomial of degree $m$ and let $\mathcal{C}$ be a cyclic code of length $n=2^m-1$ with the generating polynomial
\begin{equation}
h(x) = \frac{x^n-1}{f(x)}.
\end{equation}
In other words,
\begin{equation}
\mathcal{C} = \{\c \in GF(2)^n\mid h(x)|\c(x)\},
\end{equation}
where
\begin{equation}
\c(x) = \sum_{i=0}^{n-1}c_ix^i.
\end{equation}
When $f(x)$ is primitive, as in our case, a code constructed in such a way is usually referred to as a simplex code. All the non-zero codewords of the obtained code are shifts of each other and are called \textit{Golomb sequences} \cite{golomb1967shift} (we can, therefore, simply say that a simplex code is generated by a Golomb sequence).

Let $\mathcal{C}$ be the simplex code constructed from the primitive binary polynomial $f(x)$ as before. Fix a non-zero codeword $\varphi \in \mathcal{C}$ (a Golomb sequence) and construct a real symmetric matrix
\begin{equation}
\label{eq:def_a_eq}
\A_n = \{a_{ij}\}_{i,j=0}^{n-1} = \bigg\{\frac{1}{2\sqrt{n}}(-1)^{\varphi(i-j) + \varphi(j-i)}\bigg\}_{i,j=0}^{n-1}.
\end{equation}
Matrix $\A_n$ can be interpreted in the following way. Consider a circulant non-symmetric matrix

\begin{equation}
\T = \begin{pmatrix} 
\varphi(0) & \varphi(1) & \varphi(2) & \dots & \varphi(n-1) \\
\varphi(n-1) & \varphi(0) & \varphi(1) & \dots & \varphi(n-2) \\
\varphi(n-2) & \varphi(n-1) & \varphi(0) & \dots & \varphi(n-3) \\
\vdots & \vdots & \vdots & \ddots & \vdots \\
\varphi(1) & \varphi(2) & \varphi(3) & \dots & \varphi(0) \\
\end{pmatrix}.
\end{equation}
The consecutive rows of $\T$ are simply cyclic shifts of the Golomb sequence written in its first rows. The symmetric matrix $\A_n$ can now be written as
\begin{equation}
\A_n = \frac{1}{2\sqrt{n}}\zeta(\T+\T^\top).
\end{equation}
It is easy to check that the obtained matrix is circulant, since for any $k \in [n],\; \A_n$ is invariant under the shift of indices of the form
\begin{equation}
i \to i+ k \mod n,\quad j \to j+k \mod n.
\end{equation}

Recall that any non-zero codeword of $\mathcal{C}$ is a cyclic shift of $\varphi$, therefore, we may obtain an ensemble of matrices from the code $\mathcal{C}$ indexed by integers within the range $a \in [n]$, as
\begin{equation}
\A_n(a) = \bigg\{\frac{1}{2\sqrt{n}}(-1)^{\varphi(i-j+a) + \varphi(j-i+a)}\bigg\}_{i,j=0}^{n-1},
\end{equation}
with the original matrix $\A_n$ corresponding to $\A_n(0)$.
\begin{definition}
Given a primitive binary polynomial $f(x)$, an ensemble of pseudo-random matrices $\mathcal{A}_n$ of order $n$ is the set of all $\A_n(a),\; a \in [n]$ and their negatives, endowed with the uniform probability measure.
\end{definition}
Below, whenever expectation over $\A_n$ is considered it should be always treated with respect to the uniform measure over $\mathcal{A}_n$. 

\begin{definition}[\cite{speicher2009free}]
We say that a sequence $\{\A_n\}_{n=1}^\infty$ of matrices of growing sizes has an asymptotic eigenvalue distribution if
\begin{equation}
\label{eq:spec_exist}
\beta_r = \lim_{n \to \infty} \tr{\A_n^r}
\end{equation}
exist for all $r \in \N$.
\end{definition}

One of the central results of \cite{soloveychik2018symmetric} reads as follows.
\begin{prop}[Proposition 1 from \cite{soloveychik2018symmetric}]
\label{cor:main}
Let $\A_n \in \mathcal{A}_n$, then for a fixed $r \in \N$ and $n = n(m)$ tending to infinity,
\begin{equation}
\mathbb{E}\[\beta_r(\A_n)\] = \begin{cases} \beta_r + O\(\frac{1}{n}\), & r \text{ even}, \\ \qquad 0, & r \text{ odd},\end{cases}
\end{equation}
where
\begin{equation}
\label{eq:dc_mom}
\beta_r = \int x^r dF_{sc} = \begin{cases} \;\; 0, & r \text{ odd}, \\ \frac{C_{r/2}}{2^r}, & r \text{ even},\end{cases}
\end{equation}
are the moments of the semicircular distribution \cite{wigner1955characteristic} and
\begin{equation}
\label{eq:catal_num}
C_r = \frac{(2r)!}{r!(r+1)!}
\end{equation}
are the Catalan numbers.
\end{prop}
This result in particular implies that the limiting spectral law of our pseudo-random matrices is Wigner's semicircular law.

Consider a pair of sequences of matrices $\{\A_n\}_{n=1}^\infty$ and $\{\B_n\}_{n=1}^\infty$, each of which is assumed to have an asymptotic eigenvalue distribution. Ideally, we want to understand the limiting behavior of any reasonably regular function of $\A_n$ and $\B_n$. By the method of moment this calls for investigation of the moments $\tr{\A_n^{t_1}\B_n^{s_1}\cdots \A_n^{t_k}\B_n^{s_k}}$ for natural powers $t_i$ and $s_i$. Since our pseudo-random construction yields circulant matrices, they commute and we only need to study their mixed moments of the form $\tr{\A_n^t\B_n^s}$.
\begin{definition}
Let $\{\A_n\}_{n=1}^\infty$ and $\{\B_n\}_{n=1}^\infty$ be two sequences of random matrices of growing and matching sizes having asymptotic eigenvalue distributions with the moments $\delta_r$ and $\zeta_r$ respectively. Let $t, s \in \N$, we say that $\A_n$ and $\B_n$ are asymptotically independent if
\begin{equation}
\label{eq:def_free}
\tr{\(\A_n^t-\delta_t\I\)\(\B_n^s-\zeta_s\I\)} \to 0,\; n\to\infty.
\end{equation}
\end{definition}
Note that the mode of asymptotic independence (e.g., in expectation, in probability, almost surely) is determined by the mode of convergence to zero in (\ref{eq:def_free}).

\section{Asymptotic Pseudo-Independence}
\label{sec:asympt_fr_ps}
In this section we show that two sequences of pseudo-random matrices constructed as described in Section \ref{sec:gol} from different non-reciprocal primitive polynomials are asymptotically independent in expectation, namely that they satisfy the moment condition (\ref{eq:def_free}) on average over the ensembles $\mathcal{A}_n$ and $\mathcal{B}_n$.

Given a binary polynomial $f(x)$, its reciprocal is a polynomial of the same degree defined as
\begin{equation}
\label{eq:recipr_def}
\hat{f}(x) = x^{\deg f} f\(x^{-1}\).
\end{equation}

\begin{lemma}
A reciprocal of a primitive polynomial is primitive.
\end{lemma}
\begin{proof}
The result follows directly from the properties of the primitive polynomials and the fact that if $\varepsilon$ is a root of a polynomial, $\varepsilon^{-1}$ is the root of its reciprocal.
\end{proof}

Assume that the generating polynomial $g(x)$ of the Golomb sequence $\psi$ does not coincide neither with the generating polynomial $f(x)$ of $\phi$ nor with its reciprocal $\hat{f}(x)$.

\begin{prop}
\label{prop:main_res}
Let $\{f_m(x)\}_{m=1}^\infty$ and $\{g_m(x)\}_{m=1}^\infty$ be two sequences of different and non-reciprocal primitive polynomials of degrees $m$. For $n=2^m-1$, let $\{\A_n\}_{n=1}^\infty$ and $\{\B_n\}_{n=1}^\infty$ be pseudo-random matrices constructed from $f_m$ and $g_m$ correspondingly with arbitrary seeds, then $\A_n$ and $\B_n$ are asymptotically independent on average.
\end{prop}
\begin{proof}
Our goal is to show that the expressions of the form
\begin{equation}
\label{eq:exp_mom}
\mathbb{E}\[\tr{\(\A_n^t(a)-\beta_t\I\)\(\B_n^s(b)-\beta_s\I\)}\],
\end{equation}
for all natural $t$ and $s$ converge to zero when $n$ increases. Introduce the following quantity,
\begin{align}
\label{eq:trace_av}
E &= \mathbb{E}\[\tr{\(\A_n^t(a)-\beta_t\I\)\(\B_n^s(b)-\beta_s\I\)}\] + \beta_t\beta_s \nonumber \\ 
&=\mathbb{E}\[\tr{\A_n^t(a)\B_n^s(b)}\] \\ 
&=\frac{1}{n}\mathbb{E}\[\sum_{i,j=0}^{n-1}\[\A^t\]_{ij}\[\B^s\]_{ij}\] \nonumber \\
&= \frac{1}{n^3}\sum_{a=0}^{n-1}\sum_{b=0}^{n-1} \frac{1}{2^{2r}n^r} \sum_{i_0,\dots,i_{t-1}=0}^{n-1}\sum_{\substack{j_1,\dots,j_{s-2}=0, \\ j_0 = i_{t-1}, j_{s-1} = i_0}}^{n-1} (-1)^{\gamma_{\i,\j}(a,b)}, \nonumber
\end{align}
where
\begin{equation}
r=\frac{t}{2}+\frac{s}{2},
\end{equation}
and we denote
\begin{align}
&\gamma_{\i,\j}(a,b) = \sum_{q=0}^{t-1} \varphi(i_{q+1}-i_q+a) + \varphi(i_q - i_{q+1}+a) \nonumber \\
&\;\;+ \sum_{q=0}^{s-1} \psi(j_{q+1}-j_q+b) + \psi(j_q - j_{q+1}+b),
\end{align}
where we treat the indices $q$ of the vertices $i_q$ and $j_q$ modulo $t$ and $s$, respectively. Instead of treating the expression in (\ref{eq:exp_mom}), it is more convenient to demonstrate that $E$ converges to $\beta_r\beta_s$ which is equivalent to the original statement. Let us also write explicitly the condition on indices appearing in (\ref{eq:trace_av}) as
\begin{equation}
\label{eq:crit_cond}
j_0 = i_{t-1},\quad j_{s-1} = i_0.
\end{equation}
Set
\begin{gather}
\label{eq:t_def_i}
u_q = i_{q+1}-i_q \mod n,\quad q=0,\dots,t-1, \\
w_q = j_{q+1}-j_q \mod n,\quad q=0,\dots,s-1.
\end{gather}
Denote the obtained $t$- and $s$-tuples by
\begin{gather}
\u = (u_0,\dots,u_{t-1}) \in [n]^t, \\
\w = (w_0,\dots,w_{s-1}) \in [n]^s,
\end{gather}
and following \cite{soloveychik2018symmetric} use the function
\begin{align}
\label{eq:eq_40}
&\qquad\qquad\quad \nu_t : [n]^t  \to GF(2)^n, \nonumber\\
(u_0,&\dots,u_{t-1}) \\
& \mapsto \bigg\{\sum_{q=0}^{t-1} \mathbbm{1}(u_q=i)+\mathbbm{1}(-u_q=i) \;\mmod 2 \bigg\}_{i=0}^{n-1}, \nonumber
\end{align}
where $\mathbbm{1}$ is an indicator function and the equalities are modulo $n$. We refer the reader to \cite{soloveychik2018symmetric} for a detailed discussion on the properties of $\nu_t$. Briefly, $\nu_t(\cdot)$ takes the $t$-tuple $\u = (u_0,\dots,u_{t-1})$ and first maps it into an extended $2t$-tuple $\(\u,-\u\) = \(u_0,\dots,u_{t-1},-u_0,\dots,-u_{t-1}\) \in [n]^{2t}$. Then it calculates the number of appearances of every number $u \in [n]$ in this $2t$-tuple, which we denote by $\#\{u\}$ and constructs a codeword $\c \in GF(2)^n$ by setting its elements with indices $u$ to $\#\{u\} \;\mmod 2$ and zeros otherwise. For convenience, we suppress the subscript of $\nu_t$ below.

Rewrite $\gamma_{\i,\j}(a,b)$ as
\begin{equation}
\gamma_{\i,\j}(a,b) = \tau(\nu(\u),a) + \tau(\nu(\w),b),
\end{equation}
where
\begin{align}
\label{eq:tau_def}
&\tau(\nu(\u);a) \\
&\quad= \begin{cases}\sum_{q=0}^{t-1} \big[ \varphi(u_q+a) + \varphi(-u_q+a)\big] & \mmod 2, \\ &\nu(\u) \neq \bm{0} \\ \qquad\qquad\qquad\quad 0, &\nu(\u) = \bm{0}.\end{cases} \nonumber
\end{align}

With this notation, we obtain
\begin{equation}
\label{eq:trace_av_13}
E = \frac{1}{2^{2r}n^{r+3}} \sum_{a=0}^{n-1}\sum_{b=0}^{n-1} \sum_{\u,\w} (-1)^{\tau(\nu(\u);a) + \tau(\nu(\w);b)},
\end{equation}
where we assume $\u$ and $\w$ to satisfy (\ref{eq:crit_cond}). Let us denote
\begin{equation}
k = j_0 - i_0,
\end{equation}
then (\ref{eq:trace_av_13}) can be rewritten as
\begin{equation}
E = \frac{1}{2^{2r}n^{r+3}} \sum_{a=0}^{n-1}\sum_{b=0}^{n-1}\sum_{k=0}^{n-1} \sum_{i_0=0}^{n-1}\sum_{\u_k,\w_k} (-1)^{\tau(\nu(\u);a) + \tau(\nu(\w);b)},
\end{equation}
where $t$- and $s$-tuples $\u_k$ and $\w_k$ have their elements $i_{t-1} - i_0 = k$ and $j_0 - j_{s-1} = k$, respectively. Clearly for fixed $k$ and $i_0$, the averages over $a$ and $b$ decouple and we can switch the order of summation to obtain
\begin{align}
E =  \sum_{k=0}^{n-1} \frac{1}{n}\sum_{i_0=0}^{n-1} &\[\frac{1}{2^{t}n^{t/2+1}}\sum_{a=0}^{n-1}\sum_{\u_k} (-1)^{\tau(\nu(\u_k);a)}\] \nonumber \\
\times & \[\frac{1}{2^{s}n^{s/2+1}}\sum_{b=0}^{n-1} \sum_{\w_k} (-1)^{\tau(\nu(\w_k);b)}\].
\end{align}
Now we deal with the sums in the square brackets separately. We focus on the first sum, the second is treated analogously. Let us consider the case of $k=0$ and even $t$. Here, similarly to \cite{soloveychik2018symmetric} we need to count the number of even paths starting and ending at $i_0$ in order to calculate the leading term of the expected value. The calculation follows the same reasoning as in \cite{soloveychik2018symmetric} and for every fixed $i_0$ yields
\begin{equation}
\frac{1}{2^{t}n^{t/2+1}}\sum_{a=0}^{n-1}\sum_{\u_k} (-1)^{\tau(\nu(\u_k);a)} = \beta_t + O\(\frac{1}{n}\).
\end{equation}
For all other combinations of $k > 0$ or odd $t$, using the same approach as in the derivation of a bound on III in the proof of Proposition 1 in \cite{soloveychik2018symmetric}, we get
\begin{equation}
\frac{1}{2^{t}n^{t/2+1}}\sum_{a=0}^{n-1}\sum_{\u_k} (-1)^{\tau(\nu(\u_k);a)} = O\(\frac{1}{n}\).
\end{equation}
Similarly, for the second sum,
\begin{equation}
\frac{1}{2^{s}n^{s/2+1}}\sum_{b=0}^{n-1}\sum_{\w_k} (-1)^{\tau(\nu(\w_k);a)} = \beta_s + O\(\frac{1}{n}\),
\end{equation}
when $k=0$ and $s$ is even. Otherwise,
\begin{equation}
\frac{1}{2^{s}n^{s/2+1}}\sum_{b=0}^{n-1}\sum_{\w_k} (-1)^{\tau(\nu(\w_k);a)} = O\(\frac{1}{n}\).
\end{equation}
Overall, we conclude
\begin{align}
E & = \sum_{k=0}^{n-1} \frac{1}{n}\sum_{i_0=0}^{n-1} \[\beta_t + O\(\frac{1}{n}\)\]\[\beta_s + O\(\frac{1}{n}\)\] \nonumber \\ 
&= \beta_t\beta_s + O\(\frac{1}{n}\),
\end{align}
which according to (\ref{eq:trace_av}) completes the proof.
\end{proof}

It is important to note that Proposition \ref{prop:main_res} claims asymptotic independence of the two sequences at hand on average. In fact, asymptotic almost sure independence can also be demonstrated using the same technique as in \cite{soloveychik2018symmetric} (see Figure \ref{fig:mixed_mom} showing the decay of the variance). However, to avoid duplication of the proof we decide to omit the rigorous derivation here.

\section{Numerical Experiments}
\label{sec:num}
In this section, we illustrate our theoretical results from Section \ref{sec:asympt_fr_ps} using numerical simulations. More specifically, we examine the behavior of low mixed moments of our pseudo-random matrices when the sizes of the latter grow.

Let us fix a range $M = m_b,\dots,m_e$ of integers and consider two sequences of primitive binary polynomials $f_{m_i}$ and $g_{m_i},\; m_i \in M$. Each of the constructed polynomials gives raise to an ensemble of cardinality $n_i = 2^{m_i}-1$ of pseudo-random matrices of sizes $n_i \times n_i$. Denote the corresponding ensembles by $\mathcal{A}_{n_i}$ and $\mathcal{B}_{n_i}$. In our experiment we took $m_b=7,\; m_e = 19$. Polynomials $f_{m_i}$ were chosen to be the first polynomials in the corresponding rows of the table \cite{vzivkovic1994table}. Polynomials $g_{m_i}$ were obtained through $2$-fold decimation of $f_{m_i}$-s and can be easily checked to be non-reciprocal with $f_{m_i}$-s \cite{goresky2012algebraic}. 

\begin{figure}[!t]
\includegraphics[width=3.75in]{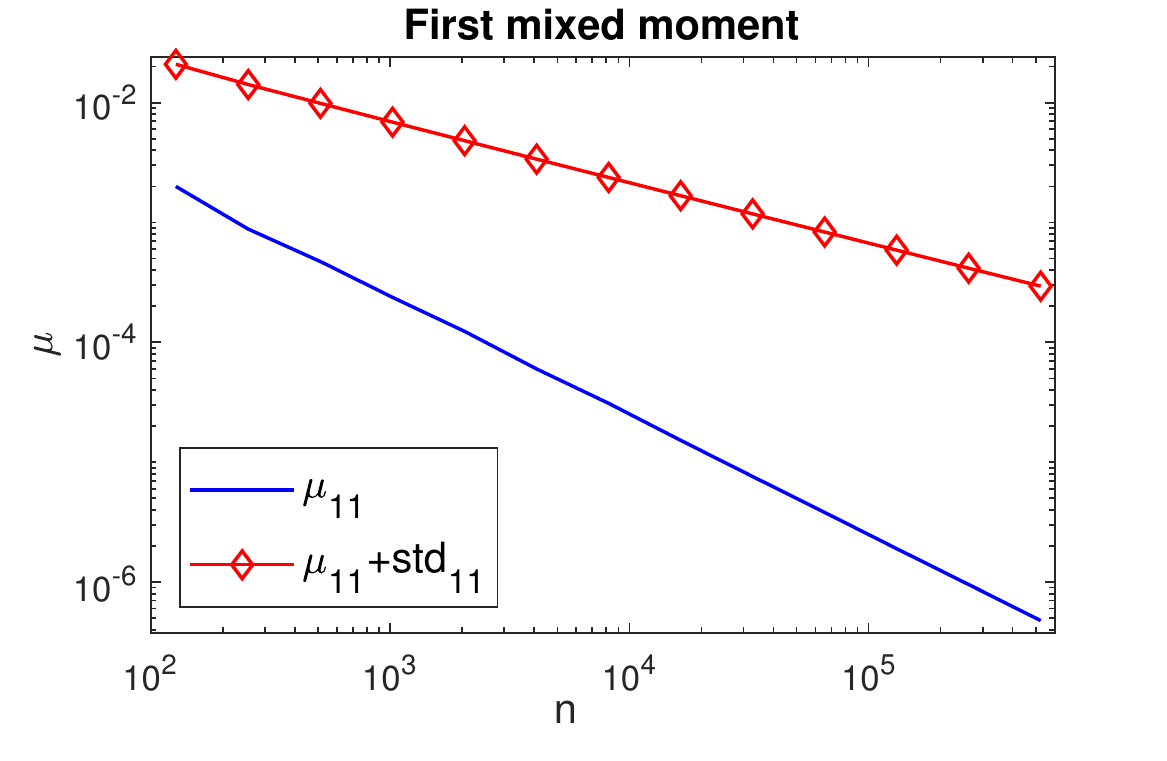}
\caption{First mixed moment of the form (\ref{eq:mix_mom}) plus its standard deviation region in pseudo-random matrices of sizes $n= 2^7-1,\dots,2^{19}-1$.}
\label{fig:mixed_mom}
\end{figure}

We focus on studying the behavior of the expected odd mixed moment
\begin{equation}
\label{eq:mix_mom}
\mu_{ts}(n_i) = \mathbb{E}_{\A_{n_i} \sim \mathcal{A}_{n_i}, \B_{n_i} \sim \mathcal{B}_{n_i}} \tr{\A_{n_i}^t\B_{n_i}^s},
\end{equation}
as a function of $n_i$. Figure \ref{fig:mixed_mom} demonstrates that the mixed moments at hand decay to zero as expected. In addition, it shows the decay of the variance of the trace in (\ref{eq:mix_mom}), which implies almost sure asymptotic independence as explained earlier.

\begin{figure}[!t]
\includegraphics[width=3.75in]{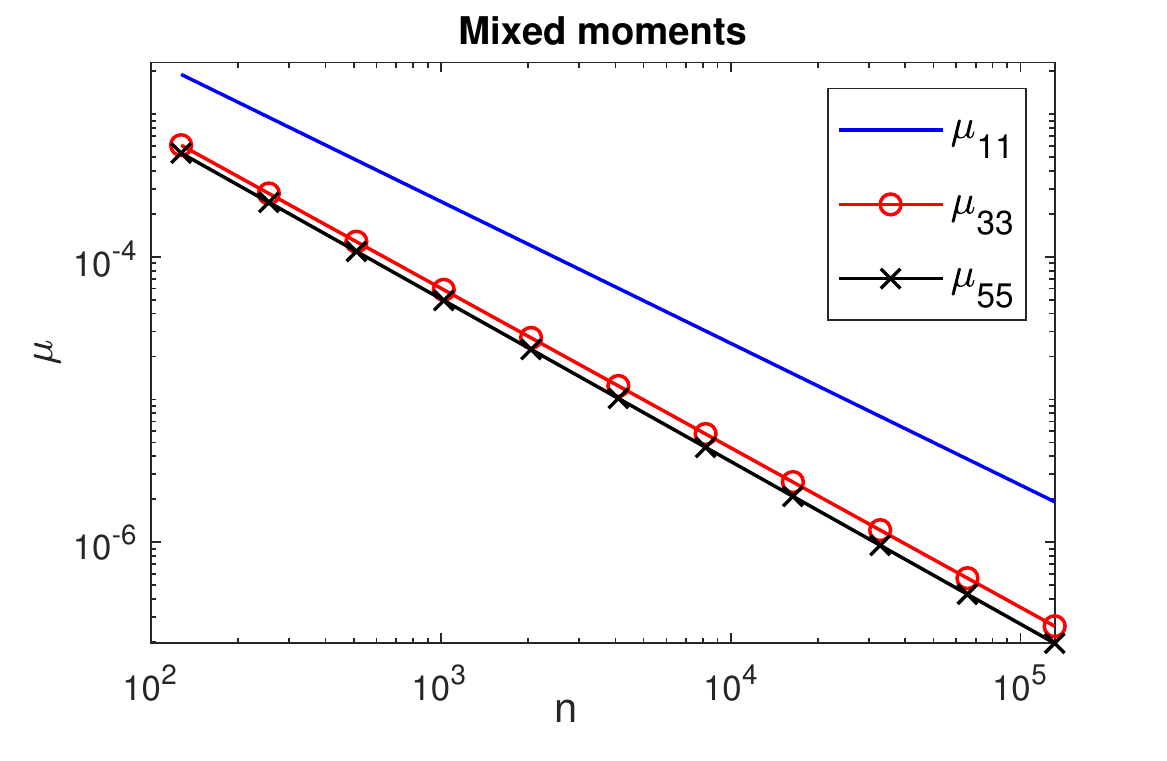}
\caption{Higher mixed moments of the form (\ref{eq:mix_mom}) in pseudo-random matrices of sizes $n= 2^7-1,\dots,2^{17}-1$.}
\label{fig:h_mixed_mom}
\end{figure}

Figure \ref{fig:h_mixed_mom} provides an empirical comparison of the rates of convergence of higher mixed moments to zero. Here, we took two polynomials $f_{m_i}^j$ and $g_{m_i}^j,\; j=1,2$ of every degree in the range defined by $m_b=7,\; m_e=17$ from the same table \cite{vzivkovic1994table} and averaged the moments over the two corresponding ensembles $\mathcal{A}_{n_i}^j$ and $\mathcal{B}_{n_i}^j$. Remarkably, this graph supports our theoretical result established in Proposition \ref{prop:main_res} claiming that mixed moments decay with the rate of $O\(\frac{1}{n}\)$.

\section{Conclusions}
\label{sec:conc}
In this article, we show that the recently discovered in \cite{soloveychik2018symmetric} family of pseudo-random symmetric sign matrices exhibits asymptotic independence properties. This results allows one to generate pairs of \textit{random-looking} symmetric sign matrices with semicircular limiting spectrum and vanishing odd mixed moments.

\bibliographystyle{IEEEtran}
\bibliography{ilya_bib}
\end{document}